\documentclass{amsart}
\title{Littlewood Complexes for Symmetric Groups}
\author{Christopher Ryba}
\usepackage{hyperref}
\usepackage{fullpage}
\usepackage{amsmath}
\usepackage{amsfonts}
\usepackage{amssymb}
\usepackage{comment}
\usepackage{ytableau}
\usepackage{subfig}
\usepackage{todonotes}
\usepackage{graphicx, caption}

\input tikz.tex
\usetikzlibrary{cd}

\DeclareMathOperator{\Hom}{Hom}

\DeclareMathOperator{\Sym}{Sym}

\begin{document}
\maketitle

\newtheorem{theorem}{Theorem}[section]
\newtheorem{lemma}[theorem]{Lemma}
\newtheorem{proposition}[theorem]{Proposition}
\newtheorem{corollary}[theorem]{Corollary}
\newtheorem{example}[theorem]{Example}
\newtheorem{remark}[theorem]{Remark}
\newtheorem{definition}[theorem]{Definition}

\begin{abstract}
We construct a complex $\mathcal{L}_\bullet^\lambda$ resolving the irreducible representations $\mathcal{S}^{\lambda[n]}$ of the symmetric groups $S_n$ by representations restricted from $GL_n(k)$. This construction lifts to $\mathrm{Rep}(S_\infty)$, where it yields injective resolutions of simple objects. It categorifies stable Specht polynomials, and allows us to understand evaluations of these polynomials for all $n$.
\end{abstract}

\tableofcontents

\section{Introduction}
\noindent
In this paper, we work over an arbitrary field of characteristic zero, $k$. Let $V$ be a vector space, and write $T(V)$ for the tensor algebra of $V$, and $\Sym(V)$ for the symmetric algebra of $V$. Because $\Sym(V)$ is a quotient of $T(V)$ as an algebra, $\Sym(V)$ is in particular a module over $T(V)$. In fact, $\Sym(V)$ is a module for $T(V)^{\otimes n}$, for any $n$, where each tensor factor acts in the usual way. The main technical result of this paper is the computation of
\[
\mbox{Tor}_{i}^{T(V)^{\otimes n}}(k, \Sym(V))
\]
as a $GL(V) \times S_n$ module (where the symmetric group $S_n$ acts by permutation of tensor factors on $T(V)^{\otimes n}$, and trivially on $\Sym(V)$). These $\mathrm{Tor}$ groups are the homology of a certain Chevalley-Eilenberg complex, whose $GL(V)$-isotypic components we call Littlewood complexes and denote $\mathcal{L}_\bullet^\lambda$. The terms of $\mathcal{L}_\bullet^\lambda$ are representations of $S_n$ restricted from $GL_n(k)$ (where $S_n$ embeds as permutation matrices), which may be thought of as an extension to $S_n$ of the work of Sam, Snowden, and Weyman on Littlewood Complexes for classical groups \cite{SamSnowdenWeyman}. Although there are substantial similarities to their work (e.g. computing $\mathrm{Tor}$ groups with Koszul complexes), there are also significant differences (e.g. working with noncommutative algebras).
\newline \newline \noindent
For large $n$, the (co)homology of $\mathcal{L}_\bullet^\lambda$ is the Specht module $\mathcal{S}^{\lambda[n]}$ (where $\lambda[n]$ is the partition obtained by appending a part of size $n-|\lambda|$ to the start of $\lambda$) concentrated in degree zero. The complex $\mathcal{L}_\bullet^\lambda$ is defined for all $n \in \mathbb{Z}_{\geq 0}$ (in fact also $n = \infty$), and the (co)homology is either zero, or given by an irreducible representation of $S_n$ concentrated in a single degree, which can be computed using an algorithm similar to the Borel-Weil-Bott theorem.
\newline \newline \noindent
The complexes $\mathcal{L}_\bullet^\lambda$ yield injective resolutions of simple objects in the category $\mathrm{Rep}(S_\infty)$ studied by Sam and Snowden in \cite{SamSnowden2} (but see also \cite{SamSnowden1}). Additionally, the Euler characteristic of the complex recovers an identity of Assaf and Speyer \cite{AssafSpeyer} about stable Specht polynomials (these are the same as the irreducible character basis of Orellana and Zabrocki \cite{OZ2}), and allows us to determine how these polynomials behave for all $n$ (in particular, when $n$ is small).
\newline \newline \noindent
The structure of the paper is as follows. In Section 2, we recall relevant background material. In Section 3, we express some $\mathrm{Tor}$ groups computed in \cite{SamSnowdenWeyman} combinatorially. We then compute the $\mathrm{Tor}$ groups mentioned in the introduction in Section 4. In Section 5 we define the Littlewood complex $\mathcal{L}_\bullet^\lambda$, compare it to the version for classical groups in \cite{SamSnowdenWeyman}, and explain its connection to $\mathrm{Rep}(S_\infty)$. In Section 6, we explain how $\mathcal{L}_\bullet^\lambda$ categorifies stable Specht polynomials, and hence determine evaluations of stable Specht polynomials for small $n$.
\newline \newline \noindent
{\bf Acknowledgements.}
The author would like to thank Steven Sam and Andrew Snowden for helpful conversations.

\section{Background}
\noindent
We recall some facts about representations of symmetric groups and general linear groups in characteristic zero, as well the Koszul and Chevalley-Eilenberg complexes.
\newline \newline \noindent
A partition $\lambda = (\lambda_1, \lambda_2, \ldots)$ is a weakly decreasing sequence of nonnegative integers with finitely many nonzero entries. We call $\lambda_i$ the parts of the partition, and say that the length of $\lambda$ is $l(\lambda) = \max\{i \mid \lambda_i \neq 0\}$, the number of nonzero parts of $\lambda$. The size of $\lambda$ is $|\lambda| = \sum_i \lambda_i$, the sum of the parts. If $|\lambda| = n$, we say that $\lambda$ is a partition of $n$, and write $\lambda \vdash n$. Additionally, given a partition $\lambda$, we may define the dual (or transpose) partition $\lambda^\prime$ by ${\lambda^\prime}_i = |\{j \mid \lambda_j \geq i\}|$.
\newline \newline \noindent
The irreducible representations of the symmetric group $S_n$ over $k$ are the Specht modules $\mathcal{S}^\lambda$, indexed by partitions $\lambda \vdash n$. In particular $\mathcal{S}^{(n)}$ is the trivial representation, and $\mathcal{S}^{(1^n)}$ is the sign representation; we have that $\mathcal{S}^{(1^n)} \otimes \mathcal{S}^\lambda = \mathcal{S}^{\lambda^\prime}$. On the other hand, given a vector space $V$ over $k$, the irreducible polynomial representations of the general linear group $GL(V)$ are given by Schur functors $\mathbb{S}^\lambda(V)$ for $l(\lambda) \leq \dim(V)$. Here,
\[
\mathbb{S}^\lambda(V) = \left(V^{\otimes n} \otimes \mathcal{S}^\lambda \right)^{S_n},
\]
where the action of $S_n$ on $V^{\otimes n}$ is by permutation of tensor factors. In particular, $\mathbb{S}^{(n)}(V)$ and $\mathbb{S}^{(1^n)}(V)$ are the $n$-th symmetric and exterior powers of $V$ respectively. One formulation of Schur-Weyl duality asserts that the functor
\[
SW_n(-) = \left(V^{\otimes n} \otimes - \right)^{S_n}
\]
is an equivalence of categories between degree $n$ polynomial representations of $GL(V)$ and the category of representations of $S_n$, provided that $\dim(V) \geq n$. 
\newline \newline \noindent
Given a Lie algebra $\mathfrak{g}$, the Lie algebra homology with coefficients in a $\mathfrak{g}$-module $M$ is defined as
\[
H_i(\mathfrak{g},M) = \mbox{Tor}_{i}^{U(\mathfrak{g})}(k, M),
\]
where $U(\mathfrak{g})$ is the universal enveloping algebra of $\mathfrak{g}$, and $k$ is the trivial $\mathfrak{g}$-module. The Lie algebra homology may be computed with the Chevalley-Eilenberg complex, whose $r$-th term is
\[
\bigwedge\nolimits^r (\mathfrak{g}) \otimes M,
\]
we omit the differential because we will not need it. Details can be found in \cite{weibel}.
\newline \newline \noindent
Let $A = \bigoplus_{i \geq 0} A_i$ be a $\mathbb{Z}_{\geq 0}$-graded algebra such that $A_0 = k$, and each $A_i$ is finite dimensional over $k$. Let $A^+ = \oplus_{i > 0} A_i$ be the positively graded part of $A$. Then $k = A/A^+$ is a graded $A$-module. We say that $A$ is Koszul if the induced (``internal'') grading on the Ext algebra
\[
\mathrm{Ext}_A^{\bullet}(k,k)
\]
coincides with the homological grading.
\newline \newline \noindent
A Koszul algebra $A$ is necessarily a quadratic algebra (i.e. it is generated in degree one with relations in degree two):
\[
A = T(A_1)/(R),
\]
where $R \subseteq A_1 \otimes A_1$. We may define the Koszul dual algebra $A^!$ to be
\[
A^! = T(A_1^*)/(R^\perp),
\]
where $R^\perp \subseteq A_1^* \otimes A_1^*$ is the set of elements orthogonal to $R$ under the natural pairing. Note that $A^!$ is also a graded algebra (and $A^{!!} = A$).
\newline \newline \noindent
Koszul algebras admit a free resolution of $k$ (as a graded $A$-module) called the Koszul complex, whose $r$-th chain module is
\[
A \otimes (A_r^!)^* = \Hom_k(A_r^!, A).
\]
When we view the elements as linear functions from $A_r^!$ to $A$, the differential is given by
\[
d(f) = \sum_{i} x_i f(- \cdot x_i^*),
\]
where $\{x_i\}$ is a basis of $A_1$, and $\{x_i^*\}$ is the dual basis of $A_1^*$. The tensor product of Koszul algebras is again Koszul, and the associated Koszul complex is the tensor product of the individual Koszul complexes.
\newline \newline \noindent
Two particular examples will be useful for us. The first is that $\Sym(V)$ and $\bigwedge (V^*)$ are Koszul dual, yielding the Koszul complex
\[
\cdots \to \Sym(V) \otimes \bigwedge\nolimits^2(V)  \to \Sym(V) \otimes V \to \Sym(V) \otimes k \to 0,
\]
where the differential is given by $d = \sum_{i} v_i \otimes v_i^*$, where $v_i$ is a basis of $V$ (acting by multiplication on $\Sym(V)$), and $v_i^*$ is the dual basis of $V^*$ (acting on $\bigwedge(V)$ by contraction). The second example is the tensor algebra $T(V)$, and the algebra $k \oplus V^*$, where $k$ has degree zero and $V^*$ has degree one. This gives the complex
\[
0 \to T(V) \otimes V \to T(V)  \to 0,
\]
where the map in question is simply multiplication (viewing $V$ as the degree one graded part of $T(V)$). Finally, we remark that both of these complexes are $GL(V)$-equivariant and also they are well defined even if $V$ is infinite dimensional (although the corresponding algebras are not Koszul), because the complex is the direct limit of the corresponding complexes for finite dimensional subspaces of $V$ (direct limits preserve exactness).
\section{Some Tor Groups}
\noindent
Let $U, V, W$ be vector spaces, and let $d = \dim(U)$. Say that a pair of partitions $(\lambda, \mu)$ is admissible if $l(\lambda) + l(\mu) \leq d$, in which case $(\lambda_1, \ldots, \lambda_r, 0, \ldots, 0, -\mu_s^\prime, \ldots, \mu_1^\prime)$ (where $r = l(\lambda)$ and $s = l(\mu)$) is a well-defined weight of $GL(U)$, whose associated highest-weight irreducible representation we denote $\mathbb{S}^{[\lambda, \mu]}(U)$. For example, when $\mu$ is the empty partition, we have $\mathbb{S}^{[\lambda, \mu]} = \mathbb{S}^\lambda(U)$.
\newline \newline \noindent
Consider the algebras $A = \Sym(V \otimes W)$ and $B = \Sym(V \otimes U \oplus W \otimes U^*)$. Note that $B$ is an $A$-algebra via the canonical map
\[
V \otimes W \to V\otimes U \otimes W \otimes U^* \subseteq \Sym^2(V \otimes U \oplus W \otimes U^*),
\]
hence $B$ is in particular an $A$-module. Corollary 5.15 of \cite{SamSnowdenWeyman} states that as a representation of $GL(V) \times GL(W) \times GL(U)$,
\[
\mathrm{Tor}_{i}^{A}(B, k) = 
\bigoplus_{i_d(\lambda, \mu) = i} \mathbb{S}^\lambda(V) \otimes \mathbb{S}^\mu(W) \otimes \mathbb{S}^{[\tau_d(\lambda, \mu)]}(U),
\]
where the quantities $i_{d}(\lambda, \mu)$ and $\tau_{d}(\lambda, \mu)$ are defined via the following recursion (see Subsection 5.4 of \cite{SamSnowdenWeyman}, but note that the notation is slightly different).
\begin{definition}
If $(\lambda, \mu)$ is admissible, set $i_d(\lambda, \mu) = 0$ and $\tau_d(\lambda, \mu) = (\lambda, \mu)$. Otherwise, consider the Young diagrams of $\lambda$ and $\mu$. Let $R_\lambda$ and $R_\mu$ be the border strips of length $l(\lambda) + l(\mu) - d - 1$ starting at the intersection of the first column and final row of $\lambda$ and $\mu$ respectively, if they exist. If both exist, are nonempty, and $\lambda \backslash R_\lambda$ and $\mu \backslash R_\mu$ are partitions, put $\tau_d(\lambda, \mu) = \tau_d(\lambda \backslash R_\lambda, \mu \backslash R_\mu)$, and
\[
i_d(\lambda, \mu) = c(R_\lambda) + c(R_\mu) -1 + i_d(\lambda \backslash R_\lambda, \mu \backslash R_\mu),
\]
where $c(R)$ indicates the number of columns that the border strip $R$ intersects. If either border strip fails to exist, or is empty, or either $\lambda \backslash R_\lambda$ or $\mu \backslash R_\mu$ is not a partition, set $i_d(\lambda, \mu) = \infty$, and leave $\tau_d(\lambda, \mu)$ undefined.
\end{definition}

\noindent
The case $d=1$ will be important to us; we will characterise the functions $i_1$ and $\tau_1$. For this, we need a special case of the Bott algorithm.
\begin{definition}
Given a partition $\lambda$, and $n \in \mathbb{Z}$, we define $\delta_n(\lambda)$ and $\lambda[n]$ as follows. Let $r \geq l(\lambda)+1$, and consider the vector $v = (n - |\lambda|, \lambda_1, \ldots, \lambda_{r-1})$ (where $\lambda_i = 0$ if $i > l(\lambda)$). Let $\rho = (r-1, r-2, \ldots, 0)$ be the Weyl vector. If $v + \rho$ has no repeated entries, there is a unique permutation $w \in S_r$ such that $w(v + \rho)$ has decreasing entries. If $w(v + \rho) - \rho$ has nonnegative entries, it defines a partition, and we set $\lambda[n] = w(v + \rho) - \rho$ and $\delta_n(\lambda) = l(w)$ (using the usual length function on $S_r$). If $w(v + \rho) - \rho$ has a negative entry, or if $v + \rho$ has a repeated entry, take $\delta_n(\lambda) = \infty$ and leave $\lambda[n]$ undefined.
\end{definition}
\begin{remark}
We make a few observations.
\begin{enumerate}
\item This definition is independent of the value of $r$. 
\item When $n$ is sufficiently large, $(n-|\lambda|, \lambda_1, \ldots, \lambda_{r-1}) + \rho$ is decreasing, so $\delta_n(\lambda) = 0$, and $\lambda[n]$ is the partition obtained by appending a long top row to $\lambda$ so that the total size of the resulting partition is $n$. 
\item If we disregard the first entry of $v + \rho$, the entries are already ordered, so $\delta_{n}(\lambda)$ counts the number of indices $i$ such that $\lambda_i +r - (i+1) > n- |\lambda|+r-1$, which reduces to $\lambda_i - i > n - |\lambda|$.
\end{enumerate}
\end{remark}
\noindent
When $d=1$, a weight is an integer $n$. We restrict ourselves to the case $n \geq 0$, so in an admissible pair of partitions $(\lambda, \mu)$, we take $\mu$ to be the trivial partition.
\begin{theorem} \label{combi}
We have that $\tau_1(\lambda, \mu)$ is defined and equal to $n$ if and only if there is a partition $\nu$ such that $\lambda = \nu[n + 
|\nu|]$ and $\mu = \nu^\prime$. In this case, $i_1(\lambda, \mu) = |\nu| - \delta_{n + |\nu|}(\lambda)$.
\end{theorem}
\begin{proof}
We recall some facts about a  partition $\nu$. Let $r, s \in \mathbb{Z}_{\geq 0}$ such that $r \geq l(\nu)$ and $s \geq l(\nu^\prime)$ (we take them to be sufficiently large so that subsequent calculations are well defined). We write $\rho_m = (m-1, m-2, \ldots, 0)$ for the Weyl vector of size $m$.
\begin{enumerate}
\item The disjoint union 
\[
\{(\nu + \rho_r)_i\}_{i=1}^r
\coprod
\{r+s-1 - (\nu^\prime + \rho_s)_j\}_{j=1}^s
\]
is equal to $\{0, 1, \ldots, r+s-1\}$ (see 1.7 in Chapter 1 Section 1 of \cite{Macdonald}).

\item Let $\sigma$ be obtained from $\nu$ by adding a border strip of size $p$. Then $\sigma + \rho_r$ is obtained from $\nu + \rho_r$ by adding $p$ to an entry, and rearranging the entries in decreasing order. (See Example 8 (a) in Chapter 1 Section 1 or \cite{Macdonald}.) Moreover, suppose the border strip starts in the $i$-th row and ends in the $j$-th row, with $i>j$. The $i$-th entry of $\nu + \rho_r$ was the one that $p$ was added to, and after reordering, it became the $j$-th entry of $\sigma + \rho_r$.

\item If $R$ is a border strip, and $r(R)$ is the number of rows that $R$ intersects, we have the equation $|R| = r(R) + c(R) - 1$. (This is clear when $|R| = 1$, and incrementing the size of $|R|$ extends $R$ into a new row or column.)

\item If $\nu[n + |\nu|]$ is defined, then $l(\nu[n + |\nu|]) = l(\nu) + 1$.
\end{enumerate}
\noindent
The proof is by induction on the number of steps of the recursion defining $\tau_n(\lambda, \mu)$ required to reach the base case of an admissible pair. The base case of our induction is zero steps, in which case the statement holds with $\nu$ being the trivial partition.
\newline \newline \noindent
Now let us consider how border strips may be removed (in accordance with the recursion). To prove the ``if'' direction, suppose we are given $(\nu[n + |\nu|], \nu^\prime)$. According to the algorithm, we must remove border strips of size
\[
|R| = l(\nu[n+|\nu|]) + l(\nu^\prime)-1-1 = l(\nu) + \nu_1 - 1.
\]
In the case of $\nu^\prime$, this is the length of the unique largest border strip (it intersects every row and column of $\nu^\prime$ by fact (3)). By fact (2), removing this border strip amounts to subtracting $|R|$ from the first entry of $\nu^\prime + \rho_s$, turning it from $\nu_1^\prime + s -1 = l(\nu) + s-1$ into $s - \nu_1$. We use fact (1) to turn this into a statement about $\nu$ rather than $\nu^\prime$. 
\newline \newline \noindent
The second part of the set partition of $\{0,1, \ldots, r+s-1\}$ contained
\[
r+s-1 - (l(\nu) + s-1) = r - l(\nu)
\]
before removing the border strip, while afterwards, that element was replaced by
\[
r+s-1 - (s - \nu_1) = \nu_1 + r - 1.
\]
Consequently, the first part of the set partition must have contained $\nu_1 + r -1$ before removing the border strip, and $r - l(\nu)$ afterwards. Indeed, this can be achieved by subtracting $|R|$ from the first entry of $\nu + \rho_r$. Let us write $\sigma$ for the partition obtained by removing this border strip from $\nu$. We must check that removing a particular border strip of length $|R|$ turns $\nu[n + |\nu|]$ into $\sigma[n + |\sigma|]$. We do this by verifying that adding the border strip to $\sigma[n + |\sigma|]$ yields $\nu[n + |\nu|]$.
\newline \newline \noindent
The entries of $\sigma[n +|\sigma|] + \rho_{r+1}$ are (in some order)
\[
n + r, r - l(\nu), \nu_2 + r - 2, \nu_3 + r - 3, \ldots, \nu_r.
\]
We check that adding the border strip by adding $|R|$ to the entry $r - l(\nu)$ is permitted by the algorithm. For this to be the case, the border strip must start in the first column and the bottom row of the resulting partition. This means that it must be constructed by adding $|R|$ to an entry corresponding to a part of $\sigma[n +|\sigma|]$ of size zero. We check that the entry $r - l(\nu)$ satisfies this requirement. Note that
\[
\nu_{l(\nu)} + r - l(\nu) > r - l(\nu) > \nu_{l(\nu)+1}+r-(l(\nu) +1) = r - l(\nu) - 1
\]
as $\nu_{l(\nu)} >0$ and $\nu_{l(\nu)+1} = 0$ by definition. Because $n \geq 0$, $n+r > r - l(\nu)$, so it follows that $r - l(\nu)$ is smaller than $n+r$ and exactly $l(\nu)-1$ other entries. Hence after sorting, the $l(\nu)$-th entry is $r - l(\nu)$. When we subtract $(\rho_{r+1})_{l(\nu)} = r - l(\nu)$, we obtain zero, as required.
\newline \newline \noindent
To complete the proof of the ``if'' direction, we must check that $i_n(\lambda, \mu)$ behaves as claimed. The amount by which $i_n(\nu[n+|\nu|], \nu^\prime)$ changes (in this recursion step) is
\[
c(R_{\nu[n+|\nu|]}) + c(R_{\nu^\prime}) - 1
=
|R| - r(R_{\nu[n+|\nu|]}) + c(R_{\nu^\prime})
=
|R|-(r(R_{\nu[n+|\nu|]}) - l(\nu)),
\]
where we use fact (3) and the fact that $R_{\nu^\prime}$ intersects all $l(\nu)$ columns of $\nu^\prime$. Note that $|R|$ is precisely $|\nu| - |\sigma|$, so it suffices to check that $\delta_{n + |\nu|}(\nu) - \delta_{n + |\sigma|}(\sigma) = r(R_{\nu[n+|\nu|]}) - l(\nu)$. Recall that $R_{\nu[n + |\nu|]}$ was added to $\sigma[n + |\sigma|]$ by adding $|R|$ to the entry of $\sigma[n + |\sigma|]+\rho_{r+1}$ at index $l(\nu)+1$ (which was equal to $r - l(\nu)$). This entry became $\nu_1 + r - 1$ which is either the second largest entry (if $n+r > \nu_1 + r -1$), or the largest entry (if $n+r < \nu_1 + r - 1$). Fact (2) tells us that $r(R_{\nu[n+|\nu|]})-l(\nu)$ is equal to $0$ if $n+r > \nu_1 + r-1$ (i.e. $n > \nu_1 - 1$), and equal to $1$ if $n+r < \nu_2 + r - 1$ (i.e. $n < \nu_1 - 1$). However, $\delta_{n + |\nu|}(\nu)$ counts the number of $i$ such that $\lambda_i - i > n$, while $\delta_{n + |\sigma|}(\sigma)$ counts only the number of such $i$ with $i > 1$. The difference is therefore $0$ if $n > \lambda_1 - 1$ and $1$ if $n < \lambda_1 -1$, as required. 
\newline \newline \noindent
The proof of the ``only if'' part is essentially the same, and so we omit it.
\end{proof}

\section{Calculating the Homology}
\noindent
In this section, we prove our main theorem.
\begin{theorem}
Let $A = T(V)^{\otimes n}$ as before. Then,
\[
\mathrm{Tor}_i^A(k, \Sym(V)) = \bigoplus_{|\lambda| - \delta_n(\lambda) = i} \mathbb{S}^{\lambda^\prime}(V) \otimes \mathcal{S}^{\lambda [n]}
\]
as representations of $GL(V) \times S_n$.
\end{theorem}

\begin{proof}
Because $T(V)$ is Koszul with $T(V)_1 = V$, it follows that $A$ is also Koszul, with generating set $A_1 = V^{\oplus n} = V \otimes k^n$. We will write $v^{(r)}$ to indicate the vector $v \in V$ inside the $r$-th summand. The Koszul complex for $A$ is the $n$-th tensor power of the Koszul complex for $T(V)$:
\[
0 \to T(V)\otimes V \to T(V) \to 0.
\]
Writing the total space of this complex as $T(V) \otimes (k \oplus V)$, where $T(V)$ and $k$ are in degree zero and $v$ is in degree one, the $n$-th tensor power is
\[
A \otimes (k \oplus V)^{\otimes n}.
\]
The differential is defined as follows. Let $\{v_i\}$ be a basis of $V$, and $\{v_i^*\}$ be the dual basis of $V^*$. Then,
\[
d = \sum_{r, i} v_i^{(r)} \otimes {v_i^*}^{(r)},
\]
where $(v_i^*)^{(r)}$ acts on the $r$-th tensor factor of $(k\oplus V)^{\otimes n}$ by annihilating $k$, and mapping $V$ to $k$ in the natural way.
\newline \newline \noindent
To calculate $\mathrm{Tor}_i^A(k, \Sym(V))$, we apply the functor $- \otimes_A \Sym(V)$ to the projective resolution. We obtain the chain groups
\[
\Sym(V) \otimes
(k \oplus V)^{\otimes n} 
\]
and differential
\[
d = \sum_{i} v_i \otimes \left( \sum_r {v_i^*}^{(r)} \right),
\]
because $v^{(r)} \in A$ acts on $\Sym(V)$ by multiplying by $v$.
\newline \newline \noindent
Fix an auxiliary vector space $W$ of dimension at least $n$, and let us apply the Schur-Weyl functor $\left(W^{\otimes n} \otimes -\right)^{S_n}$ to the complex. This turns it from a complex of $GL(V) \times S_n$-modules into a complex of $GL(V) \times GL(W)$-modules. 
We obtain
\begin{eqnarray*}
& & \Sym(V) \otimes \left( W^{\otimes n} \otimes (k \oplus V)^{\otimes n} \right)^{S_n} \\
&=&
\Sym(V) \otimes \left((W \oplus W \otimes V)^{\otimes n}\right)^{S_n} \\
&=&
\Sym(V) \otimes \Sym^n(W \oplus W \otimes V) \\
&=&
\Sym(V) \otimes \bigoplus_i \Sym^{n-i}(W) \otimes \Sym^i (W \otimes V),
\end{eqnarray*}
where $W \otimes V$ is an odd superspace, so its $i$-th symmetric power is equal to the $i$-th exterior power of the underlying vector space. We may therefore write the $i$-th chain group as
\[
\Sym(V) \otimes \bigoplus_i \Sym^{n-i}(W) \otimes \bigwedge \nolimits^i (W \otimes V).
\]
We observe that the action of $\sum_r {v_i^*}^{(r)}$ becomes ``contracting with $v_i^*$'', meaning
\[
\sum_{j} w_j \otimes (w_j^* \otimes v_i^*),
\]
where $w_j$ is a basis of $W$ and $w_j^*$ is the dual basis of $W^*$. Thus the differential becomes
\[
\sum_{i,j} v_i \otimes w_j \otimes (w_j^* \otimes v_i^*).
\]
At this point it is convenient to consider all $n$ simultaneously. Let us take the direct sum of these complexes over all $n$, which has the effect of replacing $\Sym^{n-i}(W)$ by $\Sym(W)$. After identifying $\Sym(W) \otimes \Sym(V) = \Sym(W \oplus V)$, we recognise this as the Koszul complex computing
\[
\mathrm{Tor}_{i}^{\Sym(W \otimes V)}(k, \Sym(W \oplus V)),
\]
where the action of $w \otimes v \in \Sym(W \otimes V)$ on $\Sym(W \oplus V)$ is my multiplying by $wv$. This is the $d=1$ case of what is computed by \cite{SamSnowdenWeyman} as discussed in Section 3. Applying Theorem \ref{combi} gives
\[
\mathrm{Tor}_{i}^{\Sym(W \otimes V)}(k, \Sym(W \oplus V))
=
\bigoplus_{m \in \mathbb{Z}_{\geq 0}}
\bigoplus_{|\nu| - \delta_{m + |\nu|}(\nu) = i}
\mathbb{S}^{\nu^\prime}(V) \otimes
\mathbb{S}^{\nu[m + |\nu|]}(W).
\]
We now undo Schur-Weyl duality, which has the effect of picking out only the Schur functors of degree $n$ in $W$, and converting them back in to representations of $S_n$. This immediately yields the result.
\end{proof}

\section{A Littlewood Complex for Symmetric Groups}
\noindent
Recall that the tensor algebra $T(V)$ is the universal enveloping algebra of the free Lie algebra $L(V)$. Therefore $A = T(V)^{\otimes n}$ is the universal enveloping algebra of $L(V)^{\oplus n} = L(V) \otimes k^n$. So, we have computed the Lie algebra homology
\[
H_i(L(V)\otimes k^n; \Sym(V)) = \mathrm{Tor}_i^A(k, \Sym(V)).
\]
Because the action of $S_n$ on $L(V)^{\oplus n} = L(V) \otimes k^n$ by permutation of summands (equivalently, permutation on $k^n$) preserves the Lie algebra structure, this passes to an action of $S_n$ on the Lie algebra homology (as before, $S_n$ acts trivially on $\Sym(V)$). Our proof relied on computing the relevant Koszul complex, but we could have also considered the Chevalley-Eilenberg complex. Here, the chain groups are
\begin{eqnarray*}
& &\bigwedge\nolimits^i (L(V) \otimes k^n) \otimes \Sym(V) \\
&=& \bigoplus_{\mu \vdash i} \mathbb{S}^{\mu^\prime}(L(V)) \otimes \mathbb{S}^{\mu}(k^n) \otimes \Sym(V),
\end{eqnarray*}
where the action of $S_n$ on $\mathbb{S}^{\lambda}(k^n)$ is implicitly restricted from $GL_n(k)$ (embedded as permutation matrices).
\newline \newline \noindent
By applying $\Hom_{GL(V)}(\mathbb{S}^{\lambda^\prime}(V), -)$, this provides a resolution of $\mathcal{S}^{\lambda[n]}$ by modules restricted from $GL_n(k)$. If we define the multiplicity space
\[
M(\lambda, \mu) = 
\Hom_{GL(V)}(\mathbb{S}^{\lambda^\prime}(V), \mathbb{S}^{\mu^\prime}(L(V)) \otimes \Sym(V)),
\]
we may write the $i$-th chain module as
\[
\bigoplus_{\mu \vdash i} M(\lambda, \mu) \otimes \mathbb{S}^{\mu}(k^n).
\]
Because $L(V)$ is equal to $V$ plus terms of higher degree as representations of $GL(V)$, $\mathbb{S}^{\mu^\prime}(L(V))$ is isomorphic to $\mathbb{S}^{\mu^\prime}(V)$ plus terms of higher degree. Hence $M(\lambda, \mu) = 0$ unless $|\mu| \leq |\lambda|$, and when $|\mu| = |\lambda|$, $M(\lambda, \mu)$ is one-dimensional if $\mu=\lambda$ and zero otherwise. Thus our complex becomes
\[
0 \to \mathbb{S}^{\lambda}(k^n) \to \bigoplus_{\mu \vdash |\lambda| - 1} M(\lambda, \mu) \otimes \mathbb{S}^{\mu}(k^n)
\to \cdots \to
\bigoplus_{\mu \vdash 0} M(\lambda, \mu) \otimes \mathbb{S}^{\mu}(k^n)
\to 0.
\]
The homology is $\mathcal{S}^{\lambda[n]}$, concentrated in degree $|\lambda| - \delta_n(\lambda)$ (or vanishes identically, if $\delta_n(\lambda) = \infty$).
\begin{definition}
Let us instead view the above chain complex as a cochain complex (i.e. change the degree indexing from descending to ascending). We call this the Littlewood complex associated to the partition $\lambda$, and denote it $\mathcal{L}_\bullet^\lambda$.
\end{definition}
\noindent
The cohomology of $\mathcal{L}_\bullet^\lambda$ is $\mathcal{S}^{\lambda[n]}$, concentrated in degree $\delta_n(\lambda)$ (or the complex is acyclic, if $\delta_n(\lambda) = \infty$).
\newline \newline \noindent
The name ``Littlewood complex'' is taken from \cite{SamSnowdenWeyman}, where it has the following meaning. Let $V$ be a vector space, and let $G(V) \subseteq GL(V)$ be either the symplectic group on $V$ (for a choice of symplectic form), or the orthogonal group on $V$ (for a choice of orthogonal form). An irreducible representation of $G(V)$ is not typically a representation of $GL(V)$, however it turns out that that any irreducible representation of $G(V)$ may be resolved by representations that extend to $GL(V)$. The Littlewood complex $L_\bullet^\lambda$ (associated to a partition $\lambda$ defining an irreducible representation) is the minimal such resolution.
\newline \newline \noindent
Key features of the Littlewood complexes $L_\bullet^\lambda$ discussed in Section 2 of \cite{SamSnowdenWeyman} include:
\begin{enumerate}
\item One can define $L_\bullet^\lambda$ without issue even if $l(\lambda) > \dim(V)$, in which case the homology is either zero, or is equal to an irreducible representation of $G(V)$ concentrated in a single degree. There is a ``modification rule'' for determining the irreducible and the degree.
\item The modification rule is similar to the Borel-Weil-Bott theorem (it involves a dotted Weyl group action).
\item The Littlewood complex $L_\bullet^\lambda$ may be viewed as an isotypic component of a $GL(E)$-equivariant (for an auxiliary vector space $E$) Koszul complex.
\item The homology may be computed using algebraic geometry and the cohomology of certain vector bundles on Grassmannians.
\item There is a category of ``algebraic'' representations of $G(\infty)$ (i.e. the infinite symplectic group, or infinite orthogonal group), denoted $\mathrm{Rep}(G(\infty))$. There is a specialisation functor 
\[
\Gamma_V: \mathrm{Rep}(G(\infty)) \to \mathrm{Rep}(G(V)).
\]
The Littlewood complex $L_\bullet^\lambda$ computes the derived functors on the simple objects of $\mathrm{Rep}(G(\infty))$ (which are indexed by partitions $\lambda$ of any size). See Subsection 1.3 of \cite{SamSnowdenWeyman} for more details.
\end{enumerate}
In our case, the first two points apply verbatim. The third point is slightly different; our Littlewood complex is an isotypic component of a Chevalley-Eilenberg complex. In the case of an abelian Lie algebra, the Chevalley-Eilenberg complex reduces to a Koszul complex for a polynomial algebra. So our situation may be viewed as analogous in the noncommutative world. The fourth point has no clear comparison. It would be interesting to to have a ``geometric'' realisation of our Littlewood complex. There is a version of the fifth point, which we now briefly explain.
\newline \newline \noindent
There is a category $\mathrm{Rep}(S(\infty))$ of ``algebraic'' representations of the infinite symmetric group. We direct the interested reader to Section 6 of \cite{SamSnowden2} for more information about this category. Objects of $\mathrm{Rep}(S(\infty))$ are subquotients of direct sums of tensor powers of $k^\infty$ (the permutation representation of $S_\infty$). Unlike its finite counterparts, this category is not semisimple.
\newline \newline \noindent
The simple objects are indexed by partitions $\lambda$ of any size, and are directed limits of $\mathcal{S}^{\lambda[n]}$ as $n \to \infty$. The objects $\mathbb{S}^{\mu}(k^\infty)$ are injective, although they are not indecomposable. We may set $n = \infty$ in our Littlewood complex, which becomes the directed limit of $\mathcal{L}_\bullet^\lambda$ for finite $n$ under the inclusions $k^n \to k^{n+1}$ defined by appending a zero entry to a vector. The cohomology is the directed limit of $\mathcal{S}^{\lambda[n]}$ in degree zero. So the $n=\infty$ Littlewood complex is an injective resolution of the simple object of $\mathrm{Rep}(S(\infty))$ indexed by $\lambda$.
\newline \newline \noindent
There is a specialisation functor
\[
\Gamma_n : \mathrm{Rep}(S(\infty)) \to \mathrm{Rep}(S_n),
\]
such that $\Gamma_n(k^\infty) = k^n$ and $\Gamma_n$ is a left-exact tensor functor. Applying this functor to the $n = \infty$ Littlewood complex yields $\mathcal{L}_\bullet^\lambda$ for finite $n$. Therefore $\mathcal{L}_\bullet^\lambda$ provides a computation of the derived specialisation functors on simple objects. This was originally done in Proposition 7.4.3 of \cite{SamSnowden1} using the indecomposable injective objects instead of $\mathbb{S}^\mu(k^\infty)$.

\section{Categorification of Stable Specht Polynomials}
\noindent
Suppose now that $n$ is sufficiently large, so that $\delta_n(\lambda) = 0$, and the cohomology of $\mathcal{L}_\bullet^\lambda$ is $\mathcal{S}^{\lambda[n]}$ in degree zero. We may compute the character of $\mathcal{S}^{\lambda[n]}$ as a representation of $S_n$ by taking the trace of a permutation matrix on the Euler characteristic of $\mathcal{L}_\bullet^\lambda$.
\newline \newline \noindent
Recall that the trace of $g \in GL_n(k)$ acting on $\mathbb{S}^{\mu}(k^n)$ is equal to the Schur function $s_\mu$ evaluated at the eigenvalues of $g$ (viewed as an $n \times n$ matrix). Taking the Euler characteristic of $\mathcal{L}_\bullet^\lambda$, we obtain the symmetric function
\[
s_\lambda^\dagger = \sum_{\mu \leq |\lambda|} (-1)^{|\lambda| - |\mu|} \dim(M(\lambda, \mu)) s_\mu.
\]
This symmetric function satisfies the property that when it is evaluated at the eigenvalues of a permutation matrix, the value is the character of $\mathcal{S}^{\lambda[n]}$ evaluated at that permutation. This is the defining property of the irreducible character basis, $\tilde{s}_\lambda$, (of the ring of symmetric functions) of \cite{OZ2}, introduced independently in \cite{AssafSpeyer} under the name stable Specht polynomials. Thus the Littlewood complexes categorify the symmetric functions $s_\lambda^\dagger$.
\newline \newline \noindent
We may express $\dim(M(\lambda, \mu))$ in terms of symmetric functions. Let $\langle -,- \rangle$ be the usual inner product on symmetric functions. Also define the two symmetric functions
\[
H = \sum_{m \geq 0} h_m, \hspace{10mm}
L = \sum_{m \geq 1} \frac{1}{m} \sum_{d \mid m} \mu(d) p_d^{m/d}
\]
to be the sum of all complete symmetric functions, and Lyndon symmetric functions, respectively (here, $\mu$ is the usual M\"{o}bius function, and $p_d$ are the power-sum symmetric functions). Thus, $H$ is the character of $\Sym(V)$ as a representation of $GL(V)$, while $L$ is the character of $L(V)$ as a representation of $GL(V)$. Finally, let us write $s_{\mu^\prime}[L]$ to indicate the plethysm of $s_{\mu^\prime}$ with $L$. Then we have:
\[
\dim(M(\lambda, \mu)) = 
\langle s_{\lambda^\prime}, s_{\mu^\prime}[L] H \rangle.
\]
After applying the Pieri rule, this is precisely Theorem 2 of \cite{AssafSpeyer}.
\newline \newline \noindent
In both \cite{OZ2} and \cite{AssafSpeyer}, these symmetric functions were defined by considering representations of $S_n$ when $n$ is sufficiently large with respect to $\lambda$. Our description allows us to understand how they behave for all $n$.
\begin{theorem}
The value of $s_\lambda^\dagger$ evaluated on the eigenvalues of a permutation matrix of size $n$ is equal to $(-1)^{\delta_n(\lambda)}$ times the character of $\mathcal{S}^{\lambda[n]}$ if $\delta_n(\lambda)$ is finite, and zero otherwise.
\end{theorem}
\begin{proof}
This is immediate from the equality of $s_\lambda^\dagger$ with the Euler characteristic of the Littlewood complex, together with the characterisation of the cohomology of the Littlewood complex.
\end{proof}


\bibliographystyle{alpha}
\bibliography{ref2.bib}

\begin{thebibliography}{SSW13}

\bibitem[AS20]{AssafSpeyer}
Sami Assaf and David Speyer.
\newblock Specht modules decompose as alternating sums of restrictions of schur
  modules.
\newblock {\em Proceedings of the American Mathematical Society},
  148(3):1015--1029, 2020.

\bibitem[Mac95]{Macdonald}
I~.~G. Macdonald.
\newblock {\em {Symmetric functions and Hall polynomials}}.
\newblock Oxford mathematical monographs. Clarendon Press New York, Oxford,
  second edition, 1995.

\bibitem[OZ16]{OZ2}
Rosa Orellana and Mike Zabrocki.
\newblock Symmetric group characters as symmetric functions.
\newblock {\em arXiv preprint arXiv:1605.06672}, 2016.

\bibitem[SS15]{SamSnowden2}
Steven~V Sam and Andrew Snowden.
\newblock Stability patterns in representation theory.
\newblock In {\em Forum of Mathematics, Sigma}, volume~3. Cambridge University
  Press, 2015.

\bibitem[SS16]{SamSnowden1}
Steven Sam and Andrew Snowden.
\newblock Gl-equivariant modules over polynomial rings in infinitely many
  variables.
\newblock {\em Transactions of the American Mathematical Society},
  368(2):1097--1158, 2016.

\bibitem[SSW13]{SamSnowdenWeyman}
Steven~V Sam, Andrew Snowden, and Jerzy Weyman.
\newblock Homology of littlewood complexes.
\newblock {\em Selecta Mathematica}, 19(3):655--698, 2013.

\bibitem[Wei95]{weibel}
Charles~A Weibel.
\newblock {\em An introduction to homological algebra}.
\newblock Number~38. Cambridge university press, 1995.

\end{thebibliography}

\end{document}